\documentclass[11pt,a4paper]{article}
\usepackage{amsfonts,amsmath,amsthm}
\usepackage{amssymb}
\usepackage[utf8]{inputenc}
\usepackage[english]{babel}
\usepackage{graphicx}
\usepackage{epstopdf}
\usepackage{tikz}
\usetikzlibrary{shapes,arrows,backgrounds}
\usepackage{subcaption}
\usepackage{array}
\usepackage{url}
\usepackage{booktabs}

\newtheorem{theorem}{Theorem}[section]
\newtheorem{lemma}[theorem]{Lemma}
\newtheorem{proposition}[theorem]{Proposition}
\newtheorem{corollary}[theorem]{Corollary}

\newtheorem{definition}[theorem]{Definition}

\theoremstyle{remark}
\newtheorem{remark}[theorem]{Remark}

\tikzstyle{if} = [diamond, draw, fill=gray!20, 
    text width=6em, text badly centered, node distance=3cm, inner sep=0pt]
\tikzstyle{block} = [rectangle, draw, fill=gray!20, 
    text width=18em, text centered, rounded corners, minimum height=4em]
\tikzstyle{group} = [rectangle, draw, text width=34em, text centered, rounded corners, minimum height=16.5em, dashed]
\tikzstyle{note} = [rectangle,text width=2em, text centered, rounded corners, minimum height=1em, font=\bfseries]
\tikzstyle{line}=[draw,->, >=latex', shorten >=1pt, thin]
\tikzstyle{io} = [trapezium, trapezium left angle=70, trapezium right angle=110, minimum width=2em, minimum height=1em, text centered, text width=16em, draw=black, fill=gray!20]
\tikzstyle{cloud} = [draw, ellipse,fill=gray!20, node distance=3cm,
    minimum height=2em]

\def\D2{\mathcal{D}_2}
\def\DI{\mathcal{D}_\infty}

\def\E2{\mathcal{E}_2}
\def\EI{\mathcal{E}_\infty}

\def\F{\mathcal{F}}
\def\G{\mathcal{G}}
\def\FF{\widetilde{\F}}
\def\FFk#1{\FF_{k_{#1}}}
\def\GG{\widetilde{\G}}

\def\R{\mathbb R}
\def\e{\varepsilon}

\title{A New Algorithm to Fit Exponential Decays} 

\author{J.~A.~F.~Torvisco \thanks{jfernandck@alumnos.unex.es}, M.~R.~Arias \thanks{arias@unex.es Corresponding author} and J.~Cabello~Sánchez \thanks{coco@unex.es} \\ Facultad de Ciencias, Universidad de Extremadura,\\ Avda. de Elvas s/n, 06006 Badajoz. Spain}

\date{}

\begin{document}

\maketitle

\begin{abstract}
This paper deals with some nonlinear problems which exponential and biexponential decays are involved in. A proof of the quasiconvexity of the error function in some of these problems of optimization is presented. This proof is restricted to fitting observations by means of exponentials having the form $f(t)=\lambda_1\exp(kt)+\lambda_2$. Based on its quasiconvexity, we propose an algorithm to estimate the best approximation to each of these decays. Besides, this algorithm does not require an initial guess.  
\end{abstract}


\section{Introduction}
Exponential decays are involved in a wide class of processes. Some of these are radioactive processes, cooling processes, dumped oscilations in charging processes and phenomena which arise from the superposition of purely periodic processes whose periods do not have integer ratios such as brightness fluctuations of variable stars or ebb and flood tide, see~\cite[p.~355-363]{willers_practical}. Exponential decays also appear in studies of viscoelastic materials. In oversimplified Maxwell and in Kelvin-Voigt models the exponential decay can be seen and it occurs in studies of stress relaxation, a typical behaviour of viscoelastic materials (see \cite[chap.~2]{marques_maxwell}). A broad collection of references about problems involving exponential decays can be consulted in the introduction of \cite{hokanson_numerically}. 

For the last two centuries, many efforts have been made to fit observations related with these problems with sums of exponential functions \cite{hokanson_numerically,Holmstrom}, \cite[p.~369-371]{whittaker_calculus}. J.~M.~Hokanson explains in the introduction of \cite{hokanson_numerically} how these efforts can be clasified in two groups. Namely, some of them are related to Prony's methods \cite{householder_prony} and other ones to least square fitting such as Levenberg-Marquardt method.

The aim of this paper is to propose a method to fit some observations by means of exponential decays. The way to achieve this goal will be to locate and estimate the minimum of a real function just by sampling it. Obviously, we have needed to study which properties guarantee this will work. Equally evident is that the function to study should verify, at least, one of these properties. Maybe this would be sufficient to assert that the method we propose can not be enclosed within any of previously referred groups. Convexity is a common lifesaver in these situations; for us, quasiconvexity played this rôle. Two kindnesses of our method are a lack of an initial guess for the approximation along with a quite reasonably time of processing.

The function $\EI$, which measures the error of some approximations, is defined on the first paragraph of Section \ref{s_max_norm}, where we prove its quasiconvex character. Also, conditions for the existence and uniqueness of the absolute minimum for $\EI$ are determined. Section \ref{s_TAC} is devoted to present the algorithm we propose, named TAC.  In sections \ref{s_aj_cooling} and \ref{s_aplications_biexp}, TAC has been used to fit data coming from a cooling process and a from study of cells' stress relaxation, respectively.

In this paper, $\R^+$ denotes $\{x\in\R : x>0\}$ and $\R^-$, $\{x\in\R : x<0\}$. We will always assume $\R^n$ and every subset endowed with its \textit{usual topology}, i.e., the one induced by the Euclidean norm. 

Throughout this paper $T=(T_1,\ldots,T_n)$ will denote a vector of $\R^n$ produced by $n$ observations of a given observable. 

The forerunner of this paper has been to give a satisfactory answer to some questions about a problem of approximation. In that problem, $T$ will be the element to be approximated; $\G$, the family of aproximants; and $\|\cdot\|_2$, the approximation criteria. $\G$ is defined as
\begin{equation*}
\G=\{(\lambda_1 e^{kt_1}+ \lambda_2,\cdots,\lambda_1 e^{kt_n}+ \lambda_2):k \in \R^-, \lambda_1, \lambda_2 \in \R\};
\end{equation*}
being $t_1<\ldots<t_n\in\R$ the instants in which the observations of $T$ were taken. 

For any $k \in \R^-$, $G_{k}$ will denote the linear plane generated by 
\begin{equation*}
E=(e^{k t_1},\cdots,e^{k t_n}) \text{ and }  I=(1, \dots, 1).
\end{equation*}
$G_{k}$ is a finite-dimensional linear subspace of the Hilbert space $(\R^n, \|\cdot\|_2)$. 
Therefore, the existence and uniqueness of the best approximation to $T$ in $G_{k}$ is a well known problem \cite[section 7.4, p.~276]{blum1972numerical}. Namely, the best approximation to $T$ in $G_{k}$ is $\lambda_1E+\lambda_2 I$, where the values of $\lambda_1 \text{ and } \lambda_2$ can be obtained by solving the linear equations system
\begin{equation*}
\left\lbrace
\begin{aligned}
\lambda_1 \langle E,E \rangle  + \lambda_2 \langle I,E \rangle & =  \langle T,E \rangle \\
\lambda_1 \langle E,I \rangle \;  + \lambda_2 \, \langle I,I \rangle & =  \langle T,I \rangle 
\end{aligned}
\right.
\end{equation*}
and $\langle \cdot,\cdot \rangle $ denotes the usual inner product in $\R^n$. The solution of this system is given by 
\begin{equation} \label{e_expresiones_lambdas}
\lambda_1(k) = \frac{\langle I,I\rangle \langle T, E\rangle -\langle E,I\rangle\langle T, I\rangle}{\langle I,I\rangle \langle E, E\rangle -\langle E,I\rangle\langle E, I\rangle},\;
\lambda_2(k) = \frac{ \langle E, E\rangle \langle T, I\rangle -\langle E,I\rangle\langle T, E\rangle}{\langle I,I\rangle \langle E, E\rangle -\langle E,I\rangle\langle E, I\rangle}.
\end{equation}
Since $\G$ is neither a linear space nor a convex or closed set, we do not have a priori results about the existence and uniqueness of best approximations. We are, however, going to determine a way to estimate the best approximation.

To do so, we need to define a couple of auxiliary functions, $\D2$ and $\E2$, later we are going to discuss some desirable properties of $\E2$.

Let us consider $\G$ as $\bigcup_{k \in \R^-} G_{k}$. For each $k \in \R^-$, let $\F_k$ denote the unique best approximation to $T$ in $G_{k}$.

We define the functions $\D2$ and $\E2$ as follows
\begin{equation*}
\begin{array}{ccl}
\D2:\R^- & \longrightarrow & \hspace{1.5mm} (\R^n, \|\cdot\|_2)\\
\quad k & \longmapsto & \D2(k):=\F_k-T,
\end{array}
\end{equation*}
and $\E2 = \|\cdot\|_2 \circ \D2$.

To guarantee the existence of a best approximation of $T$ in $\G$ we would want to prove that $\E2$ has a certain property: the quasiconvexity. We introduce now its definition.

\begin{definition}\label{d_quasiconvex}
Let $X$ be a linear topological space. A function $f:X\to\R$ is said to be quasiconvex whenever it satisfies
\begin{equation*}
f(\lambda x + (1-\lambda) y) \leq max\{f(x), f(y)\}, \quad \forall x,y \in X \text{ and } \lambda \in (0,1).
\end{equation*} 
We will say that $f$ is strictly quasiconvex when the inequality is strict.
\end{definition}

To the best of our knowledge, the first time this concept was introduced was in  \cite[p.~1554]{harvey1971review}. Some results and properties related to quasiconvex  functions can be found in, for example, \cite{harvey1971review,merentes2010el}. 

It is clear that a strictly quasiconvex function cannot have relative minima which are not absolute. Proving that $\E2$ is a strictly quasiconvex function, the problem of approximation would be solved; since every relative minima of this function is absolute. Quasiconvexity, then, allows us to construct an algorithm to find such absolute minimum or minima by sampling the function.

We tried to prove that $\E2$ is a quasiconvex function, but every attempt was a failure. We have not even found general conditions on $T$ that could make $\E2$ quasiconvex. Because of that, we decided to change the norm, using $\|\cdot\|_\infty$ instead of $\|\cdot\|_2$.

\section{Quasiconvexity of $\EI$} \label{s_max_norm}

Throughout this section we will consider $\R^n$ endowed with $\|\cdot\|_\infty$. We must now define two new functions that will play the rôle of $\D2$ and $\E2$:
\begin{equation*}
\begin{array}{ccl}
\DI:\R^- & \longrightarrow & \hspace{1.5mm} (\R^n, \|\cdot\|_\infty)\\
\quad k & \longmapsto & \DI(k):=\F_k-T,
\end{array}
\end{equation*}
and $\EI = \|\cdot\|_{\infty} \circ \DI$, i.e., $\EI(k)=\|\F_k-T\|_\infty$.

For $\DI$ to be well-defined, we need to prove the existence and uniqueness of a best approximation in every $G_k$. This was not necessary before, but now, since $(\R^n, \|\cdot\|_{\infty})$ is not a Prehilbertian space, we need to ensure it.

Therefore, we will: 
\begin{enumerate}
\item Prove the existence and uniqueness of the best approximation in each $G_k$. \label{enum_en_cada_plano}
\item Prove the quasiconvexity of $\EI$ and determine conditions for existence and uniqueness of a best approximation of $T$ in $\G$. \label{enum_en_todos_los_planos}
\end{enumerate}

\subsection{Existence and uniqueness of the best approximation in each $G_k$}

\begin{lemma}\label{existenciaJ}
Let $T\in\R^n, F\subset\R^n$ a closed subset. Then, there exists at least one $x\in F$ such that $\|T-x\|_\infty=\inf\{\|T-y\|_\infty:y\in F\}$.
\end{lemma}
\begin{proof}
Suppose $T\in\R^n$, and $F\subset\R^n$ is a nonempty closed subset that does not contain $T$, being trivial the case $T\in F$. Take $z\in F$ and $\alpha=\|T-z\|_\infty$. It is pretty clear that
$$
\inf\{\|T-y\|_\infty:y\in F\}=\inf\{\|T-y\|_\infty:y\in F, \|T-y\|_\infty\leq\alpha\},
$$
and the last expression is the infimum of a continuous function on a compact set, so this infimum is attained at some $x\in F$.
\end{proof}

\begin{remark}
The problem of finding the best approximation to $T$ in the (linear) plane $G_k=\langle E,I\rangle$ is equivalent to finding the vector with smallest norm in the (affine) plane $H=G_k-T=\{v-T:v\in G_k\}$, we will say that such a vector is {\em minimal} in $H$. For the sake of clarity, we will consider the latest way of stating the problem in the following result. 
\end{remark}

\begin{proposition}\label{rrrx}
Let $x\in H$ be such that $\|x\|_\infty=\min\{\|y\|_\infty:y\in H\}=r$. Then, there exist indices $1\leq i<j<m\leq n$ such that $x_i=-x_j=x_m=\pm r$. 

Moreover, if $x$ fulfills this condition, then it is minimal and it is the unique  minimal element in $H$. 
\end{proposition}

\begin{proof}
Suppose $x$ is minimal and let $r=\|x\|_\infty$. By the very definition of the $\max$-norm, there exists some $i\in\{1,\ldots,n\}$ such that $x_i=\pm r$. Let $r^+=\max\{x_1,\ldots,x_n\}$ and $r^-=\min\{x_1,\ldots,x_n\}$. If $r^+\neq -r^-$, then we may suppose that $r^+= -r^-+\alpha $ for some $\alpha >0$, being the opposite case very similar. Obviously $\|x-\alpha /2\cdot{I}\|_\infty=r-\alpha /2<r$ and $x-\alpha /2\cdot{I}\in H$, so we have a contradiction that shows that $r^+=-r^-$: there must exist $i, j$ such that $x_i=-x_j=\pm r$. 

Suppose that $x_i=-x_j=r$, $|x_m|<r$ for every $m\not\in\{i, j\}$ and take $v\in G_k$ such that $v_i<0, v_j>0$ --it is pretty clear that there exists such a $v$. Now, there exists $\e>0$ such that $|x_m|\leq r-\e$ for every $m\not\in\{i, j\}$, so we may take $\delta>0$ such that
$$
|x_m+\beta v_m|<r,\ \forall\ m\not\in\{i, j\}, \beta\in[0,\delta].
$$
It is clear that we can manage to find $\beta\in[0,\delta]$ such that $|x_i+\beta v_i|<r$ and $|x_j+\beta v_j|<r.$ So, $x+\beta v\in H$ and $\|x+\beta v\|_\infty<\|x\|_\infty$. 

So, for $x$ to be minimal, there must exist another index $m$ such that $x_m=\pm r$. Suppose that $i, j$ and $m$ do not fulfill the hypothesis, i.e., all the maxima lie before (or after) every minimum. We may and do assume $i< m$, $x_i=-x_m=r,$ with $x_j<r$ for every $j>i$ and $x_j>-r$ for every $j<m$, the other case is analogous. Take $\e>0$ small enough to keep $x_i-\e E_i$ as the greatest coordinate of $y=x-\e E$ and $x_m-\e E_m$ as its smallest. As $E_l>E_{l+1}>0$ for every $l=1,\ldots,n-1$, we have $r-\max\{y_1,\ldots,y_n\}=\e E_i > \e E_m = -r-\min\{y_1,\ldots,y_n\}>0$, so 
$$
\max\{y_1,\ldots,y_n\}-\min\{y_1,\ldots,y_n\}<r+r=\max\{x_1,\ldots,x_n\}-\min\{x_1,\ldots,x_n\}.
$$ 

Now, take $\delta=\frac 12(\max\{y_1,\ldots,y_n\}+\min\{y_1,\ldots,y_n\})$ and the vector $y-\delta I$. Now it is clear that 
$$
\|y-\delta I\|_\infty=\frac 12(\max\{y_1,\ldots,y_n\}-\min\{y_1,\ldots,y_n\})<r
$$ 
and $y-\delta I\in H,$ so $x$ was not a minimal element in $H$ and we have finished the first part of the proof. 

For the inverse implication, take $x\in H$ and $1\leq i<j<m\leq n$ such that $x_i=-x_j=x_m=\|x\|_\infty$. Now, $E_i>E_j>E_m$ implies $(x+\lambda E)_i-(x+\lambda E)_j>2\|x\|_\infty$ for every $\lambda>0$ and $(x+\lambda E)_m-(x+\lambda E)_j>2\|x\|_\infty$ for every $\lambda<0$. It is clear that in both cases $\|x+\lambda E+\mu I\|_\infty>\|x\|_\infty$ for any real $\mu$. For $\lambda=0$ the latest inequality is also true for every $\mu\neq 0$, so we are done. 
\end{proof}

\begin{remark}\label{remark_constante}
The description of the best approximations provided in the above proposition will be more useful than it seems. Please observe that it implies that the best approximation will be a constant only when there are two indices of $T$ where the maximum or the minimum is attained, and some minimum (resp. maximum) must be between two maxima (resp. minima). 

\end{remark}

\subsection{Best approximation for every $k$}
\begin{remark}\label{sinconstantes}
In this subsection, we will suppose that $T$ is such that some exponential approximates it better than any constant, the other case is vacuous. Indeed, if there exist $k\in(-\infty,0), b_0\in\R$ such that 
$$
\|(b_0-T_1),\ldots,(b_0-T_n)\|_\infty\leq \|(a\exp(kt_1)+b-T_1,\ldots,a\exp(kt_n)+b-T_n)\|_\infty
$$ 
for every $a, b\in\R$ then Proposition~\ref{rrrx} ensures that there is a minimum between two maxima (o a maximum between two minima), so this will happen for {\em every} $k\in(-\infty,0)$ and we have nothing to do. 
\end{remark}

\begin{remark}
This is one of the two main advantages of working with the $\max$-norm instead of the Euclidean norm, the easy way to determine whether $T$ is a good point to be approximated or not. 
\end{remark}

Before we continue, let us see a property of the family of approximants that will be key from now on. Please observe that every triple $(a,b,k)\in\R\times\R\times\R^-$ determines not only an element $(ae^{kt_1}+b,\cdots,ae^{kt_n}+b) \in \G$ but a function, too: $ae^{kt}+b$. We can consider then the family of functions
\begin{equation*}
\GG= \{ ae^{kt}+b : (ae^{kt_1}+b,\cdots,ae^{kt_n}+b) \in \G \}.
\end{equation*} 
Obviously, evaluating each non constant function of $\GG$ in $t_1,\cdots,t_n$ determines a unique vector in $\G$. Now, our immediate goal is to show that, in some sense, the `inverse' implication also holds: every vector in $\G$ determines a triple $(a,b,k)$, which in turn determines a function in $\GG$, as long as $n \geq 3$.

\begin{remark}
Because of Remark~\ref{sinconstantes}, we will need to deal just with, say, non constant exponentials, so whenever we have $f(t)=a\exp(kt)+b$ we will assume that $a\neq 0$ for the remainder of the subsection. Please observe that, if we deal with a constant vector $T=(b,\ldots,b)$, then 
$f(t)=a\exp(kt)+b$, with $a=0,$ is a function that fulfills $f(t_i)=T_i$ for every $i=1,\ldots,n$, no matter the value of $k$.
\end{remark}

Let $f_1(t)=a_1\exp(k_1t)+b_1, f_3(t)=a_3\exp(k_3t)+b_3$. Recall that $a_1, a_3\in\R^-\cup\R^+, b_1, b_3\in\R$ and $k_1, k_3\in\R^-$. 

\begin{lemma}
If there are two different $s_1, s_2\in\R$ such that $f_1'(s_1)=f_3'(s_1)$ and $f_1'(s_2)=f_3'(s_2)$, then $a_3=a_1$ and $k_3=k_1$. If, furthermore, $f_1$ and $f_3$ agree at some point, then they agree everywhere and $b_3=b_1$.
\end{lemma}

\begin{proof}
As $f_1'(t)=a_1k_1\exp(k_1t)$ and $f_3'(t)=a_3k_3\exp(k_3t)$, our hypotheses can be rewriten as 
$$
a_1k_1e^{k_1s_1}=a_3k_3e^{k_3s_1},\quad a_1k_1e^{k_1s_2}=a_3k_3e^{k_3s_2},
$$
or equivalently,
$$
e^{(k_3-k_1)s_1} = \frac{a_1k_1}{a_3k_3} = e^{(k_3-k_1)s_2}.
$$
As $s_1\neq s_2,$ this readily implies $k_3=k_1$, so $\exp((k_3-k_1)s_1)=1$ and this implies $a_3=a_1$. So, $f_1'=f_3'$ and this means that $f_1=f_3$ if and only if they agree at some point.
\end{proof}

\begin{corollary}\label{coinciden3} $f_1$ and $f_3$ agree everywhere if any of the following conditions holds:
\begin{enumerate}
\item There exists $c_0$ such that $f_1(c_0)=f_3(c_0)$, $f_1'(c_0)=f_3'(c_0)$ and $f_1''(c_0)=f_3''(c_0)$.
\item They agree at two points, and are tangent at one of them.
\item They agree at three points.
\end{enumerate}
\end{corollary}

\begin{proof}
The first item is equivalent to the following equalities:
$$
a_1 e^{k_1c_0}+b_1 = a_3 e^{k_3c_0}+b_3, \quad k_1 a_1 e^{k_1c_0}= k_3 a_3 e^{k_3c_0} \text{ and } k_1^2 a_1 e^{k_1c_0} = k_3^2 a_3 e^{k_3c_0}.
$$
The second and third equalities together imply that $k_3=k_1$, and this implies that $a_3=a_1$. Now it is readily seen that $b_3=b_1$.

For the second one, suppose there exist $c_1< c_2$ such that $f_1(c_1)=f_3(c_1)$,  $f_1(c_2)=f_3(c_2)$ and $f_1'(c_1)=f_3'(c_1)$. Then the Rolle's theorem ensures that there is some point $s\in(c_1, c_2)$ such that $f_1'(s)=f_3'(s)$, so we have two different points where $f_1'$ and $f_3'$ agree, namely $s$ and $c_1$. By the above lemma, this implies that $f_1'=f_3'$. As $f_1(c_1)=f_3(c_1)$, we are done. The case $f_1'(c_2)=f_3'(c_2)$ is analogous.

In the third case, suppose that there exist $c_1< c_2< c_3$ such that $f_1(c_1)=f_3(c_1), f_1(c_2)=f_3(c_2)$ and $f_1(c_3)=f_3(c_3)$. Again, the Rolle's theorem ensures that there exist $s_1\in(c_1,c_2)$ and $s_2\in(c_2,c_3)$ such that $f_1'(s_1)=f_3'(s_1)$ and $f_1'(s_2)=f_3'(s_2)$. Again, the previous lemma is enough to finish the proof.
\end{proof}

\begin{lemma}\label{existenciaexp}
Let be $c_1<c_2<c_3$ and $y_1> y_2> y_3$ real numbers such that 
$$
\frac{y_2-y_1}{c_2-c_1}\neq\frac{y_3-y_1}{c_3-c_1}.
$$ Then, there exist unique $a, b$ and $k$ in $\R$ such that 
$$
a e^{kc_i}+b=y_i \text{ for } i\in\{1, 2, 3\}. 
$$ 
\end{lemma}

\begin{proof}
The previous corollary ensures the uniqueness of $a, b$ and $k$, so we just need to show their existence. 

It is easy to check that, if $c_1< c_2$ and $y_1> y_2$, then, given $k\in\R, k\neq 0$, the function $f_k(t)=(y_2-y_1) g_k(t)+ y_1$, being 
$$
g_k(t)=\frac{e^{kt}-e^{kc_1}}{e^{kc_2}-e^{kc_1}}, \forall t \in \R,
$$ 
satisfies $f_k(c_1)=y_1, f_k(c_2)=y_2$ --of course this is also true if $y_1<y_2$. 

So far we have seen that for each $k\in \R\backslash \{0\}$ there is a function $f_k$ satisfying $f_k(c_1)=y_1$ and $f_k(c_2)=y_2$. Now, taking $c_3>c_2$, we shall prove that there exists an unique $k$ satisfying $f_k(c_i)=y_i$ for $i=1, 2$ and $3$.

Let us see what happens with $f_k(c_3)$ as we vary $k$. It is clear that the function 
$$
(-\infty,0)\cup(0,\infty)\ni k\mapsto f_k(c_3)\in(-\infty,y_2)
$$
is continuous. So, we must compute the limits of $f_k(c_3)$ when $k\to-\infty,$ $k\to0^-,$ $k\to 0^+$ and $k\to\infty$. 

Since $c_1<c_2<c_3$, it is clear that $g_k(c_3)$ goes to infinity when $k\to\infty$. As $y_2<y_1$, this means that $f_k(c_3)\to-\infty$. 

It is also clear that $g_k(c_3)\to 1$ when $k\to-\infty$, so $f_k(c_3)\to y_2$. 

For $k\to 0$, the limit of $g_k(c_3)$ can be easily computed via the L'H\^opital's rule to get 
$$
\lim_{k \to 0}g_k(c_3)=
\lim_{k \to 0}\frac{c_3 e^{kc_3}-c_1e^{kc_1}}{c_2e^{kc_2}-c_1e^{kc_1}} =
\frac{c_3-c_1}{c_2-c_1}
$$
and this implies that  
$$
\lim_{k\to 0} f_k(c_3)= \frac{(y_2-y_1)(c_3-c_1)}{c_2-c_1}+y_1.
$$
So, the function $\phi:\R\to(-\infty, y_2)$ defined as 
$$
k\mapsto\phi(k)=
\begin{cases}
f_k(c_3), &\mathrm{\ for\ } k\neq 0 \\
{(y_2-y_1)(c_3-c_1)}/(c_2-c_1)+y_1, &\mathrm{\ for\ } k=0
\end{cases}
$$
is continuous; and, by the previous corollary, it is also injective. As $\phi(k)$ tends to $-\infty$ when $k\to\infty$ and to $y_2$ when $k\to-\infty$, this implies that it is strictly decreasing. Therefore, for each $y_3\in(-\infty, y_2)$ there is exactly one $k\in\R$ such that $\phi(k)=y_3$. Besides, 
$$
\frac{y_2-y_1}{c_2-c_1}=\frac{y_3-y_1}{c_3-c_1} \qquad \Leftrightarrow \qquad y_3=\frac{(y_2-y_1)(c_3-c_1)}{c_2-c_1}+y_1=\phi(0),
$$ 
so for $c_1, c_2, c_3, y_1, y_2$ and $y_3$ satisfying the hypotheses there exist unique $a, b$ and $k$ such that 
$$
a e^{kc_i}+b=y_i \text{ for } i\in\{1, 2, 3\}. 
$$ 
\end{proof}

In order to prove that $\EI$ is a quasiconvex function, we will show that for every $k_1<k_2<k_3<0$, if $\FFk1$ and $\FFk3$ are the functions determined by $\F_{k_1}$ and $\F_{k_3}$, respectively, then we can find $(a_2 e^{k_2 t_1}+ b_2,\cdots,a_2 e^{k_2 t_n}+ b_2) \in G_{k_2}$ such that, for every $t_i \in \{t_1,\cdots,t_n\}$,
\begin{equation}\label{ec_desigualdad_cada_coordenada}
\begin{split}
|a_2 e^{k_2 t_i} + b_2 - T_i| & \leq \max \{|\FFk1 (t_i)-T_i|,|\FFk3(t_i)-T_i|\}\\ & \leq \max \{\EI (k_1),\EI (k_3)\}
\end{split}
\end{equation}
is satisfied and, therefore,
\[
\|(a_2 e^{k_2 t_1} + b_2,\cdots,a_2 e^{k_2 t_n} + b_2)-(T_1,\cdots,T_n)\|_{\infty} \leq \max \{\EI (k_1),\EI (k_3)\}.
\]
Hence, as
\begin{equation*}
\EI (k_2) \leq \|(a_2 e^{k_2 t_1} + b_2,\cdots,a_2 e^{k_2 t_n} + b_2)-T\|_{\infty} 
\end{equation*}
is ensured, mixing both inequalities we get 
\begin{equation}\label{ec_desigualdad_errores}
\EI (k_2) \leq\max \{\EI (k_1),\EI (k_3)\},
\end{equation}
i.e., $\EI$ is a quasiconvex function.

By Lemma \ref{coinciden3}, $\FFk1$ and $\FFk3$ agree at most in 2 points, so we will study separately what happens if $\FFk1$ and $\FFk3$ meet each other 0, 1 or 2 times. We will show that (\ref{ec_desigualdad_cada_coordenada}), and so (\ref{ec_desigualdad_errores}), hold in every case.

Before, we shall explicit a result that ensures that a function lying between two other functions is always a better approximation than the worst of them:

\begin{lemma} \label{l_enmedio}
Let be $c_1\leq c_2\leq c_3\in\R$. Then, 
\[
|c_2-\alpha|\leq\max\{|c_1-\alpha|,|c_3-\alpha|\},\forall\alpha\in\R.
\]
Moreover, if both inequalities are strict, then the last is, too.
\end{lemma}

\begin{proof}

If $\alpha\leq c_2$, then 
\[
|c_2-\alpha|=c_2-\alpha\leq c_3-\alpha=|c_3-\alpha|\leq 
\max\{|c_1-\alpha|,|c_3-\alpha|\}.
\]
If $\alpha>c_2$, then $|c_2-\alpha|=\alpha-c_2\leq\alpha-c_1=|c_1-\alpha|\leq 
\max\{|c_1-\alpha|,|c_3-\alpha|\}$, too. 
Hence, in every possible case, the result holds. As for the moreover part, if $c_1<c_2<c_3$, then $c_2-\alpha<c_3-\alpha$ and $\alpha-c_2<\alpha-c_1$, and this implies $|c_2-\alpha|<\max\{|c_1-\alpha|,|c_3-\alpha|\}$. 

\end{proof}

\begin{corollary}\label{c_enmedio}
Let $f_1, f_2$ and $f_3$ be such that $\min\{f_1(t), f_3(t)\}\leq f_2(t)\leq\max\{f_1(t), f_3(t)\}$ for every $t\in\R$. Then 
$$
\|(T_1-f_2(t_1),\ldots,T_n-f_2(t_n))\|_\infty\leq
$$
$$
\leq\max\{\|(T_1-f_1(t_1), \ldots,T_n-f_1(t_n))\|_\infty,\|(T_1-f_3(t_1),\ldots,T_n-f_3(t_n))\|_\infty\}.
$$ 
for every $(T_1,\ldots,T_n), (t_1,\ldots,t_n)\in\R^n$. 
\end{corollary}

\begin{proof}
It is straightforward from the above lemma. 
\end{proof}

\begin{proposition}\label{f1f2f3}
Let $k_1< k_2< k_3<0$ and $\FFk1, \FFk3$ be the functions determined by the best approximations for $k_1$ and $k_3$. Then, there exist $a_2, b_2$ such that $\min\{\FFk1(t), \FFk3(t)\}\leq a_2\exp(k_2t)+b_2\leq\max\{\FFk1(t), \FFk3(t)\}$ 
for every $t\in\R$. 
\end{proposition}

Before the proof, we will explicit this consequence: 

\begin{corollary}\label{csquasiconvex}
If the hypotheses in the proposition are fulfilled, then 
\begin{equation*}
\EI(k_2)<\max\{\EI(k_1),\EI(k_3)\}.
\end{equation*} 
\end{corollary}

\begin{proof}[Proof of the corollary]
Let $f_2(t)=a_2\exp(k_2t)+b_2$ be such that 
$$
\inf\{\FFk1, \FFk3\}\leq f_2\leq\sup\{\FFk1, \FFk3\}
$$ 
and suppose $\EI(k_3)\geq\EI(k_1)$. As $f_2$ and $\FFk3$ are different exponentials, Corollary~\ref{coinciden3} ensures that they agree at most in two points. By Corollary~\ref{c_enmedio}, this implies that $|f_2(t_i)-T_i|\leq \EI(k_3)$ for every $i=1,\ldots,n$, with equality in at most two points. But Proposition~\ref{rrrx} ensures that $|\FFk2(t_i)-T_i|=\|(\FFk2(t_1)-T_1),\ldots,(\FFk2(t_n)-T_n)\|_\infty=\EI(k_2)$ at least at three points, so it is clear that $\EI(k_2)\neq\EI(k_3)$. As $\EI(k_2)\leq\|(f_2(t_1)-T_1,\ldots,f_2(t_n)-T_n)\|_\infty\leq\EI(k_3)$, we are done. 
\end{proof}

\begin{remark}
This is the other great advantage of using $\|\cdot\|_\infty$ instead of $\|\cdot\|_2$. When using the Euclidean norm, we have not been able to find conditions in $T$ ensuring that a function between two other functions will be a better approximation than any of them.
\end{remark}

\subsection{Proof of Proposition \ref{f1f2f3}}

We will split the proof of the proposition into some steps: 1.- Two cuts; 2.- One cut. Limiting case and 3.- One cut. Genuine case.

\subsubsection{Two cuts}\label{ss_2_puntos}

Let $c_1<c_2$ be such that $\FFk1(c_1)=\FFk3(c_1)$ and $\FFk1(c_2)=\FFk3(c_2)$. 

Consider $k_1<k_2<k_3<0$ and $f_2(t)=a_2\exp(k_2t)+b_2$ as is the beginning of the proof of Lemma~\ref{existenciaexp}, i.e., such that $f_2(c_1)=\FFk1(c_1), f_2(c_2)=\FFk1(c_2)$. 

As all three functions are different, Corollary~\ref{coinciden3} ensures that there are no more points where any couple of them agree and, moreover, that they are not tangent neither in $c_1$ nor in $c_2$. It is pretty clear that all three functions are strictly increasing or strictly decreasing, so we have only two possibilities: either $a_1, a_2, a_3\in(-\infty,0)$ or $a_1, a_2, a_3\in(0,\infty)$. It is readily seen that in the first case $t\to\infty$ implies $f_1(t)>f_2(t)>f_3(t)$ and in the second case we just need to let $t$ tend to $-\infty$ to have $f_1(t)<f_2(t)<f_3(t)$, so in both cases there are points where $\min\{f_1(t), f_3(t)\}< f_2(t)<\max\{f_1(t), f_3(t)\}$. 

As they are not tangent, $f_1-f_2$ and $f_3-f_2$ change their signs at $c_1$ and $c_2$, and this means that $\min\{f_1(t), f_3(t)\}< f_2(t)<\max\{f_1(t), f_3(t)\}$ for every $t\not\in\{c_1,c_2\}$.

\subsubsection*{One cut}
Let us assume $\FFk1 \equiv a_1 e^{k_1 t} + b_1$ and $\FFk3 \equiv a_3 e^{k_3 t} + b_3$ coincide just in $s_1 \in \R$.

With $a_1, a_3>0$, and for small enough $t$, $\FFk1(t) > \FFk3(t)$. The other cases are analogous. 
Since they coincide in a single point, $s_1$, $\FFk1(t) > \FFk3(t)$ $\forall \ t<s_1$, and then it 
must be one of this two possible situations:

\begin{equation}\label{ec_1_corte_solo_toca}
\FFk1(t) > \FFk3(t),  \quad t \in (s_1,\infty)
\end{equation}
or
\begin{equation}\label{ec_1_corte_la_de_verdad}
\FFk1(t) < \FFk3(t),  \quad t \in (s_1, \infty)
\end{equation}

\subsubsection{One cut. Limiting case}
This case refers to the situation described in \eqref{ec_1_corte_solo_toca}. We will show that this cannot happen by means of the following 

\begin{lemma}\label{mejoressecruzan}
Let $\FF_{k_1},\FF_{k_3}$ be the best approximations for $k_1<k_3<0$. Then, there exist $s_1, s_2\in\R$ such that $\FF_{k_1}(s_1)>\FF_{k_3}(s_1)$ and $\FF_{k_1}(s_2)<\FF_{k_3}(s_2)$.
\end{lemma}

\begin{proof}
Suppose on the contrary that $\FF_{k_1}(t)\geq \FF_{k_3}(t)$ for every real $t$. Then, both functions can agree at most at just one point. Indeed, it they agree at two points then they have to be tangent at both, and this implies, by Corollary~\ref{coinciden3}, that they are the same function. By Claim~\ref{rrrx}, there exist $i, j\in\{1,\ldots,n\}$ such that $T_{i}=\FF_{k_1}(t_i)+\EI(k_1)$ and $T_{j}=\FF_{k_3}(t_j)-\EI(k_3)$. So, as $\FF_{k_1}(t)\geq \FF_{k_3}(t)$, we have $T_i-\FF_{k_3}(t_i)\geq \EI(k_1)$ and $T_j-\FF_{k_1}(t_j)\leq -\EI(k_3)$ and at least one of the inequalities is strict because $\FF_{k_1}(t)= \FF_{k_3}(t)$ for at most one $t$. Of course, $\EI(k_3)\geq T_i-\FF_{k_3}(t_i)$ and $-\EI(k_1)\leq T_j-\FF_{k_1}(t_j)$. Mixing all the inequalities, we obtain $\EI(k_3)\geq \EI(k_1), -\EI(k_1)\leq -\EI(k_3)$. Since at least one of the inequalities is strict, we have a contradiction that finishes the proof.
\end{proof}

\subsubsection{One cut. Genuine case}
Suppose we are in the situation described by (\ref{ec_1_corte_la_de_verdad}). Assume $a_1, a_3>0$, the other cases are analogous. With this assumption, the last inequality ensures $b_1 \leq b_3$. We are going to find $a_2 \in \R^+$ and $b_2 \in \R$ such that the function $a_2 e^{k_2 t} + b_2$ agrees with $\FFk1$ and $\FFk3$ in $s_1$ and remains between them. 

Considering $b_2=(b_1+b_3)/2$, $a_2$ must satisfy 
\[
a_2 e^{k_2s_1} + b_2 = a_3 e^{k_3s_1} + b_3, 
\]
and this implies that for small enough $t$, 
\[
a_1 e^{k_1 t} + b_1 > a_2 e^{k_2 t} + b_2> a_3 e^{k_3 t} + b_3.
\] 
Now, we have to guarantee $a_1e^{k_1t}+b_1<a_2e^{k_2t}+b_2<a_3e^{k_3t}+b_3$ for big enough $t$, no matter whether $b_1=b_3$ or $b_1<b_3$. Please observe that this would be enough for ending the proof. For $b_1<b_3$, the inequality is obvious, so we suppose $b_1=b_3$. Thus, $b_1$, $b_2$ and $b_3$ agree and, therefore,
\[
a_1 e^{k_1 t} + b_1 < a_2 e^{k_2 t} + b_2
\Leftrightarrow \;  a_1 e^{(k_1-k_2) t} < a_2
\]
and
\[
a_2 e^{k_2 t} + b_2 < a_3 e^{k_3 t} + b_3 
\Leftrightarrow \; a_2 e^{(k_2-k_3) t} < a_3.
\]
Both couples of equivalent inequalities are true for big enough $t$ since $a_1$, $a_2$ and $a_3$ are strictly positive and
\begin{equation*}
 \lim_{t \to \infty} e^{(k_1-k_2) t} = \lim_{t \to \infty} e^{(k_2-k_3) t} = 0.
\end{equation*}

So, we have finished the proof of Proposition~\ref{f1f2f3} since Lemma \ref{mejoressecruzan} rules out the no-cut situation. 

\subsection{Main Result}

It is time to state properly what we have:

\begin{theorem}
Let us consider $T=(T_1,\ldots,T_n)\in\R^n, t_1<\cdots<t_n\in\R$. Then, $\EI$ is a quasiconvex function. Moreover, if $T$ does not have a maximum between two minima nor a minimum between two maxima, then $\EI$ is strictly quasiconvex. 
\end{theorem}

Please recall that, as defined in the first part of this paper (see the begining of this section and the Introduction), $\EI:(-\infty,0)\to[0,\infty)$ is the function mapping $k$ to $\|\F_k-T\|_\infty$, where $\F_k$ is the best approximation to $T$ in the plane $G_k=\langle (1,\ldots,1),(\exp(kt_1),\ldots,\exp(kt_n)) \rangle$.

\begin{proof} 
 Actually, the proof of this result is this section. As it was indicated in Remark \ref{remark_constante}, when $T$ has two indices where the maximum (or the minimum) is attained, and some minimum (resp. maximum) is between two maxima (resp. minima) then the best approximation for each $k$ will be a constant, and {\em always} the same constant. Therefore, in this case, $\EI$ will be a constant function, in particular a quasiconvex function. In any other case, note that strict quasiconvexity of $\EI$ is an inmediate consequence of Corollary \ref{csquasiconvex}. Therefore, the result is proved.
\end{proof}

Once we have seen that $\EI$ is quasiconvex, we must analyse the consequences of this.
Please recall that we are assuming that $T$ is not weird, so $\EI$ is strictly quasiconvex.
As seen in~\cite[p.~128]{crouzeix}, there are three options for $\EI$. Namely: 

\noindent 1.- It is decreasing; 2.- It is increasing and 3.- There is $k_0\in(-\infty,0)$ such that $\EI$ decreases on $(-\infty,k_0)$ and increases on $(k_0,0)$.

Of course, if the third case holds then $k_0$ is the point where the minimum is attained and this means that there is a best approximation $\lambda_1\exp(k_0 t)+\lambda_2$. Let us see, in a somehow loose way, the meaning of the other two cases.

Case 1: If $\EI$ is decreasing, then the approximations $\F_k=a_k\exp(kt)+b_k$ are better as we let $k$ tend to 0. As the slopes of all $\F_k$ must be bounded in the interval $[t_1,t_n]$, $a_kk$ must be also bounded, so the second derivative of $\F_k$ tends to 0 uniformly in $[t_1,t_n]$. This means that the limit of $\F_k$ is a line $a_0t+b_0$, that must be a better approximation than any exponential. Then, the point $T$ we are trying to approximate must have a pretty strange form: there must exist $i_1<i_2<i_3<i_4$ such that
$$T_{i_1}-(a_0t_{i_1}+b_0)=a_0t_{i_2}+b_0-T_{i_2}=T_{i_3}-(a_0t_{i_3}+b_0)=a_0t_{i_4}+b_0-T_{i_4}=\pm r,$$
where $r$ is, as usual, $\|T_1-(a_0t_{1}+b_0),\ldots,T_n-(a_0t_{n}+b_0)\|_\infty$. It is pretty clear that, for $T, a_0$ and $b_0$ fulfilling the above, there is no convex or concave function that approximates $T$ better than $a_0t+b_0$. On the other hand, if there are just three indices that fulfill the above equality, then there will be an exponential closer to $T$ than $a_0t+b_0$, as long as $a_0\neq 0$.

Case 2: If $\EI$ is increasing, then the best approximation would, say, lie at $k=-\infty$.
This is what happens if $T_1=\max\{T_1,\ldots,T_n\}$ and $T_2=\min\{T_1,\ldots,T_n\}$. In this case, with $T_i=\max\{T_2,\ldots,T_n\}$, the vector 
$$
(T_1,1/2(T_2+T_i),1/2(T_2+T_i),\ldots,1/2(T_2+T_i))
$$
is closer to $T$ than any possible approximation, and it is the limit, as $k\to-\infty$, of the evaluations in $t_1,\ldots,t_n$ of some exponentials. Namely, suppose that
$t_1=0$. Then, the limit of $(T_1-1/2(T_2+T_i))\exp(kt_j)+1/2(T_2+T_i)$, as $k$ goes to $-\infty$, is $T_1$ if $j=1$ and $1/2(T_2+T_i)$ for every $j>1$.

This kind of limit will be a best approximation when $T_1=\max\{T_1,\ldots,T_n\}$ and also there exist some $i<j$ such that both $T_i=\min\{T_1,\ldots,T_n\}$ and $T_j=\max\{T_2,\ldots,T_n\}$ are fulfilled.

Now, we are in conditions, finally, to put in order everything we know about how the behaviour of $\EI$ depends on $T$. Let $M=\max\{T_1,\ldots,T_n\}$ and $m=\min\{T_1,\ldots,T_n\}$:

\begin{enumerate}
\item If $T$ has a minimum between two maxima (or a maximum between two minima), then the best approximation for every $k$ is a constant. Namely, it is $1/2(m+M)$.
\item If $T_1=M$ and $M_2=\max\{T_2,\ldots,T_n\}<M$ is attained after some minimum, then the best approximation does not exist, but the best approximations tend to
$(M,1/2(M_2+m),1/2(M_2+m),\ldots,1/2(M_2+m))$ as $k\to-\infty$.
\item If there are some $a_0, b_0\in\R, i_1<i_2<i_3<i_4\in\{1,\ldots,n\}$ such that
$$
T_{i_1}-(a_0t_{i_1}+b_0)=a_0t_{i_2}+b_0-T_{i_2}=T_{i_3}-(a_0t_{i_3}+b_0)=a_0t_{i_4}+b_0-T_{i_4}=\pm r,
$$
where $r$ is $\|T_1-(a_0t_{1}+b_0),\ldots,T_n-(a_0t_{n}+b_0)\|_\infty$, then the best approximation does not exist, but the best approximations improve as $k\to 0$ and they approach $(a_0t_1+b_0,\ldots,a_0t_n+b_0)$.
\item In any other case, there exists $k$ such that $\F_k$ is the (unique) best possible approximation.
\end{enumerate}

\section{TAC's Flowchart} \label{s_TAC}
Once the quasiconvex character of function $\EI$ has been proven, we can propose an algorithm to estimate the minimum of this function by sampling it. The convergence of this kind of algorithm is guaranteed for any non monotonic quasiconvex function. However, we wish to extend the use of this algorithm to {\it quasiconvexish} functions, that is, functions that, not being quasiconvex, behave as follows: the absolute minimum is considerably lower than the relative ones and, also, the fall into and rise from the absolute minimum take long enough. This behaviour allows us to detect the interval where the absolute minimum lies simply by sampling the function within reason. In this sense, we will say that the absolute minimum must be deep and wide enough. Think about a doodle like the one in Figure \ref{fig:sawtooth}.
\begin{figure}[ht]
  \centering
  \includegraphics[width=0.3\linewidth]{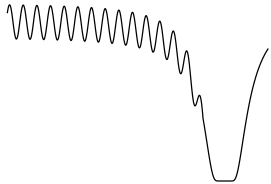}
\caption{An idea of {\it quasiconvexish}}
\label{fig:sawtooth}
\end{figure}
Even though it could never be the graph of a quasiconvex function, the absolute minimum is deep and wide enough. This situation can be observed, for example, while fitting data stemming from trigonometric functions.

It is obvious that we can not ensure the convergence of the algorithm in every possible situation. Actually, it could be quite easy to find functions where the algorithm does not work properly, due to the impossibility to guarantee that a reasonable a priori sample will include points in the interval where the function falls into and raises from the absolute minimun. That's why the strict quasiconvex character is imperative to ensure the algorithm's convergence. Nevertheless, the algorithm we propose is designed to find, or more precisely to bound, the absolute minimum of {\it quasiconvexish} functions. The algorithm proposed here is shown in  Figure \ref{tacflow}.

In case we decide to use this algorithm when the quasiconvex character of $\mathcal{E}$ is not guaranteed, as we do in Section \ref{s_aplications_biexp}, it might be advisable to add some safeguards to ensure that one finds the absolute minimum, avoiding getting stuck in a relative one. That's why block \textbf{($\blacklozenge$)} was added to the algorithm in order to densify the sample (see Figure \ref{tacflow}). This is necessary to select an interval including the absolute minimum when relative minima are detected. 

An additional condition, condition \textbf{C2} or analogous (see Figure \ref{tacflow}), must be added to the algorithm in order to improve its robustness, being able now to deal with functions which have a whole interval of absolute minima. 

This two conditions, of course, can be ignored in case we just want to fit an exponential decay as described in previous section.

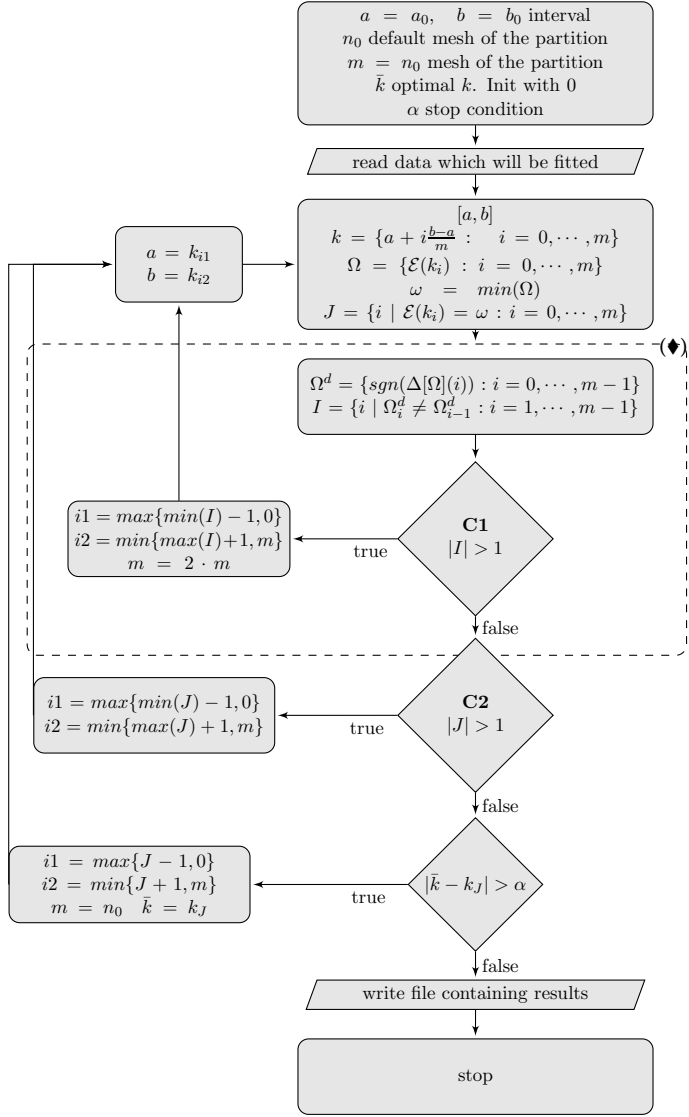
\begin{figure}[ht]
\begin{center}
\begin{tikzpicture}[node distance = 3cm, auto]
\begin{scope}[scale=0.65, transform shape]
   \node [block] (init) {$a=a_0, \quad b=b_0$ interval \\
    					  $n_0$ default mesh of the partition \\
    					  $m=n_0$ mesh of the partition\\
    					  $\bar{k}$ optimal $k$. Init with $0$\\
    					  $\alpha$ stop condition
    					 };
   \node [io, below of=init,node distance=2cm] (data) {read data which will be fitted};
   \node [block, below of=data,node distance=2.1cm] (step1) 
    	   {$[a,b]$ \\
    	    $k = \{ a + i\frac{b-a}{m} : \quad i = 0, \cdots, m \} $ \\ [0.1cm]
    	    $\Omega = \{\mathcal{E}(k_i): i=0,\cdots,m\}$ \\
    	    $\omega = min(\Omega)$\\
    	    $J = \{i \; | \; \mathcal{E}(k_i)= \omega : i=0,\cdots,m\}$ \\
    	   };
   \node [block, below of=step1,node distance=2.7cm] (step2) 
    	   {$ \Omega^d = \{sgn(\Delta[\Omega](i)): i=0,\cdots,m-1 \}$ \\
    	    $I = \{i \; | \; \Omega^d_i \neq \Omega^d_{i-1} : i=1,\cdots,m-1\}$ \\
    	   };
    \node [block, left of=step1,node distance=6cm, text width=6em] (feedback) 
    			{
    			 $a=k_{i1}$\\
    			 $b=k_{i2}$
    			};
    \node [if, below of=step2, node distance=2.9cm] (decide1) {{\bf C1}\\
                                                              $ |I|>1$
                                                             };
    \node [block, left of=decide1, node distance=6cm, text width=11em] (update1)
            {$i1 = max\{min(I)-1, 0\}$ \\
             $i2 = min\{max(I)+1, m\}$ \\
             $m=2 \cdot m$
            };    
    \node [if, below of=decide1, node distance=3.6cm] (decide2) {{\bf C2}\\
                                                                 $ |J|>1$
                                                                };
    \node [block, left of=decide2, node distance=6.5cm, text width=12em] (update2)    
            {$i1 = max\{min(J)-1, 0\}$ \\
             $i2 = min\{max(J)+ 1, m\}$
            };    
    \node [if, below of=decide2, node distance=3.45cm] (decide3) {$ |\bar{k}-k_{J}|>\alpha$};
    \node [block, left of=decide3, node distance=7cm, text width=12em] (update3)    
            {$i1 = max\{J-1, 0\}$ \\
             $i2 = min\{J+1, m\}$ \\
             $m=n_0 \quad \bar{k}=k_J$ 
            };    
    \node [io, below of=decide3, node distance=2.25cm] (data2) {write file containing results};
    \node[block, below of=data2, node distance=1.67cm] (stop) {stop};
     \begin{scope}[on background layer]
		\node[group,above of=decide2,yshift=1.4cm, xshift=-2.4cm](emp){};
	\end{scope}
	\node[note, right of=step2, node distance=4cm, yshift=1cm](nota){\textbf{($\blacklozenge$)}}; 

    \path [line] (init) -- (data);
    \path [line] (data) -- (step1);
    \path [line] (step1.south) -- (step1|-emp.north);
    \path [line] (step2) -- (decide1);
    \path [line] (decide1) -- node [near start] {true} (update1.east);
    \path [line] (update1) -- (feedback.south);
    \path [line] (decide1) -- node {false}(decide2);
    \path [line] (decide2.west) -- node [near start] {true} (update2.east);
    \path [line] (decide2) -- node {false}(decide3);
    \path [line] (decide3) -- node [near start] {true} (update3.east);
    \path [line] (decide3) -- node {false}(data2);
    \path [line] (update2.west) |- (feedback.west);
    \path [line] (update3.west) |- (feedback.west);
    \path [line] (feedback) -- (step1);
    \path [line] (data2) -- (stop);
\end{scope}
\end{tikzpicture}
\caption{TAC's Flowchart.} 
\label{tacflow}
\end{center}

\end{figure}

\subsection{The obviousness}
In this subsection we are going to verify that TAC accomplish the most elemental task. In other words, sampling a function
\begin{equation*}
f(t)= \lambda_1 e^{kt}+ \lambda_2,
\end{equation*}
TAC must find $f$.

In this case the Euclidean norm will be used instead of the max norm; although the convergence of the algorithm is not ensured. The reason is that the calculations are much simpler due to the possibility of using equations \eqref{e_expresiones_lambdas} to obtain the value of $\lambda_1$ and $\lambda_2$ for each $k$. This change will remain in following examples, stretching in that way the use of the algorithm beyond its proven convergence conditions.

Let us choose 
\begin{equation*}
f(t)=6.87654321 ~ e^{-1.12345678 t} + 2.11223344.
\end{equation*}
The function $f$ will be sampled in $t\in\{1,\dots,1000\}$. The constant $k$ is to be seek in the interval $[-10,-10^{-9}]$   and the stop condition for TAC is set as $\alpha=10^{-9}$. 

The results corresponding to this implementation of TAC are shown in Table \ref{tbfit0}.

\begin{table}[ht]
\begin{center}
\resizebox{\textwidth}{!}{
\begin{tabular}{ccccccc}
\toprule 
Divisions & CPU Time & $k$ & $\lambda_1$ & $\lambda_2$ & RSS & MSE  \\
& (in seconds) & & & & \\
\midrule
$10$ & $0.02659$ & $-1.12345678027$ & $6.87654321210$ & $2.11223344000$ & $3.00986e-39$ & $3.00986e-42$ \\
\bottomrule 
\end{tabular}
}
\caption{Result of implementation of TAC. $RSS=\sum_{i=1}^n (T_i-\F_k(t_i))^2$, being $n=1000$ the number of observations and $MSE=RSS/n$.} 
\label{tbfit0}
\end{center}
\end{table}

It is quite obvious that TAC finds $f$ according to stop condition imposed on $k$. The calculations in this section, along with the ones in the two following sections were carried out by means of a GNU Octave using an Intel Core i7-2600 3.4GHz Quad-Core processor with 4GB of RAM. The system used is an elementary OS 0.4.1 Loki (64-bit) based on a Ubuntu 16.04.3 LTS with a Linux kernel 4.4.0-93-generic.

\section{Fitting exponential decay in a Newton's law of cooling process} \label{s_aj_cooling}
The aim of this section is to show the implementation of TAC in a well known process. Consider a device submerging in the ocean to determine the immediate water temperature at some point. During the manoeuvre the sensor is recording the temperature as programmed. Since, usually, initial temperatures of the device and the water are different, the data show how thermometer and water achieve thermal balance. This behaviour is usually referred to as Newton's law of cooling. According to this law, the rate at which a body cools is proportional to the difference between the temperature of the body and the temperature of the surrounding medium; see, for example, \cite[p.~21]{dzill1996}. In other words, the time evolution of body's temperature is a solution of an ordinary differential equation, a homogeneous linear one; and therefore it must be an exponential function as follows
\begin{equation} \label{eqcooling}
P(t)=\lambda_1 e^{kt}+ \lambda_2,
\end{equation}
where $\lambda_1$ and $\lambda_2 \in \R$ and, in this case, $k<0$.

The data considered in this section were obtained during the Spanish Antarctic campaign in the Antarctic summer 2012. We implement the algorithm in order to fit, by a pattern as (\ref{eqcooling}), the records obtained by the device. This is the way we propose to estimate the water's temperature at the instant which the device was introduced in. In this application, the estimation of water's temperature would be $\lambda_2$. The results corresponding to this implementation of TAC are gathered in Table \ref{tbfit1}. A wide interval, $[-10,-10^{-9}]$, containing $k$ will be considered. The stop condition of TAC will be fixed at $\alpha=10^{-9}$.

\begin{table}[ht]
\begin{center}
\resizebox{\textwidth}{!}{
\begin{tabular}{ccccccc}
\toprule 
Divisions & CPU Time & $k$ & $\lambda_1$ & $\lambda_2$ & RSS & MSE  \\
& (in seconds) & & & & \\
\midrule
$10$ & $0.0718$ & $-0.0027323255878$ & $5.839341497$ & $-1.3650697701$ & $7.63101355666$ & $0.00139992910597$ \\
$20$ & $0.1137$ & $-0.0027323254186$ & $5.839341344$ & $-1.3650697841$ & $7.63101355667$ & $0.00139992910597$ \\
$30$ & $0.1191$ & $-0.0027323254987$ & $5.839341416$ & $-1.3650697775$ & $7.63101355666$ & $0.00139992910597$ \\ 
$40$ & $0.1398$ & $-0.0027323255060$ & $5.839341423$ & $-1.3650697769$ & $7.63101355666$ & $0.00139992910597$ \\ 
$50$ & $0.1415$ & $-0.0027323256495$ & $5.839341553$ & $-1.3650697650$ & $7.63101355668$ & $0.00139992910597$ \\
\bottomrule 
\end{tabular}
}
\caption{Result of implementation of TAC. $RSS=\sum_{i=1}^n (T_i-\F_k(t_i))^2$, being $n=5451$ the number of observations and $MSE=RSS/n$.} 
\label{tbfit1}
\end{center}
\end{table}

In the numerical analisys bibliography, relative error is defined with or without sign; in this paper we will consider the latter.

\begin{figure}[ht]
  \centering
  \begin{subfigure}[t]{.5\textwidth}
    \centering
    \includegraphics[width=\linewidth]{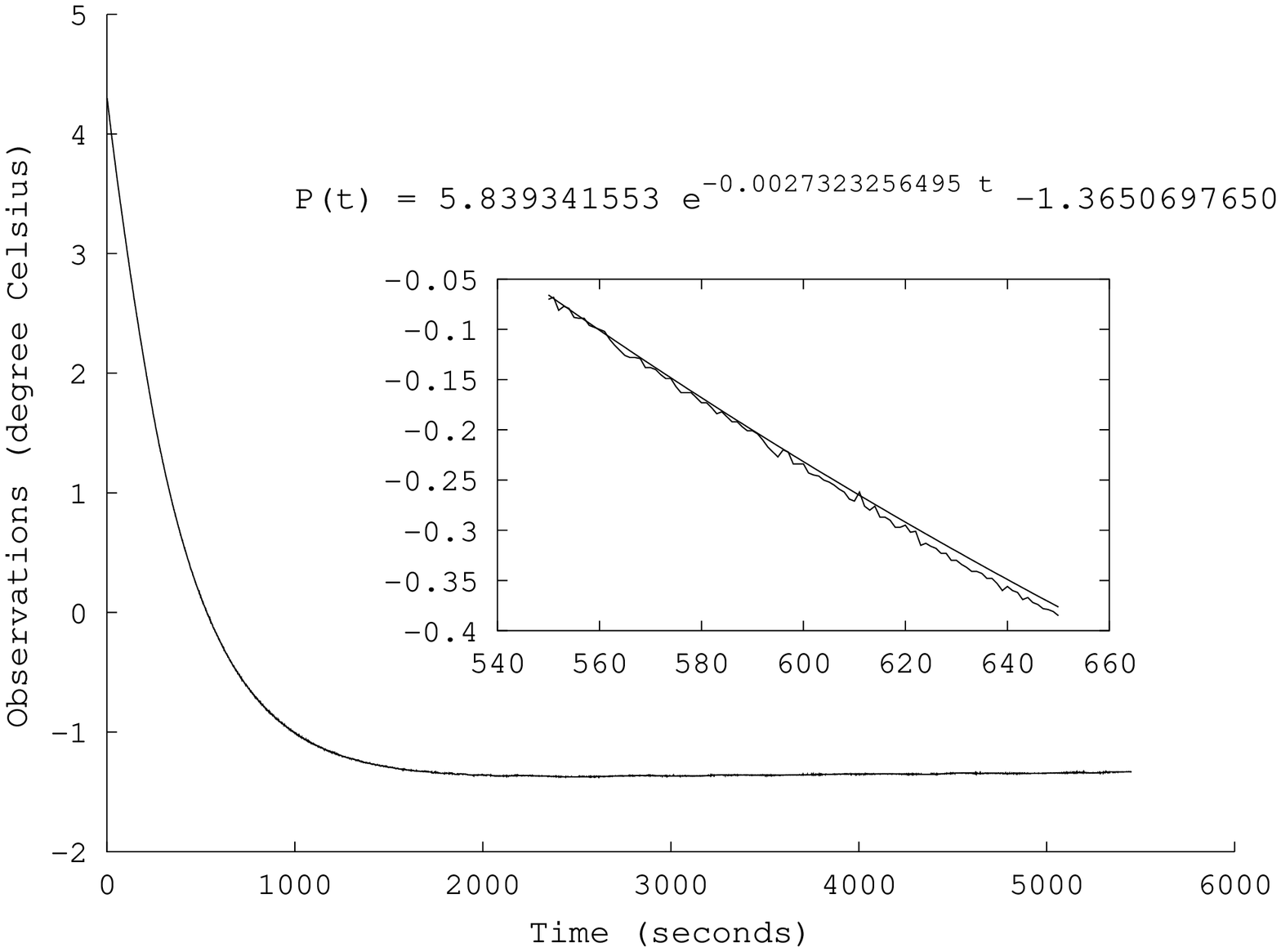}
    \caption{Data and a small detail of TAC's fit.}
  \end{subfigure}\hfill
  \begin{subfigure}[t]{.5\textwidth}
    \centering
    \includegraphics[width=\linewidth]{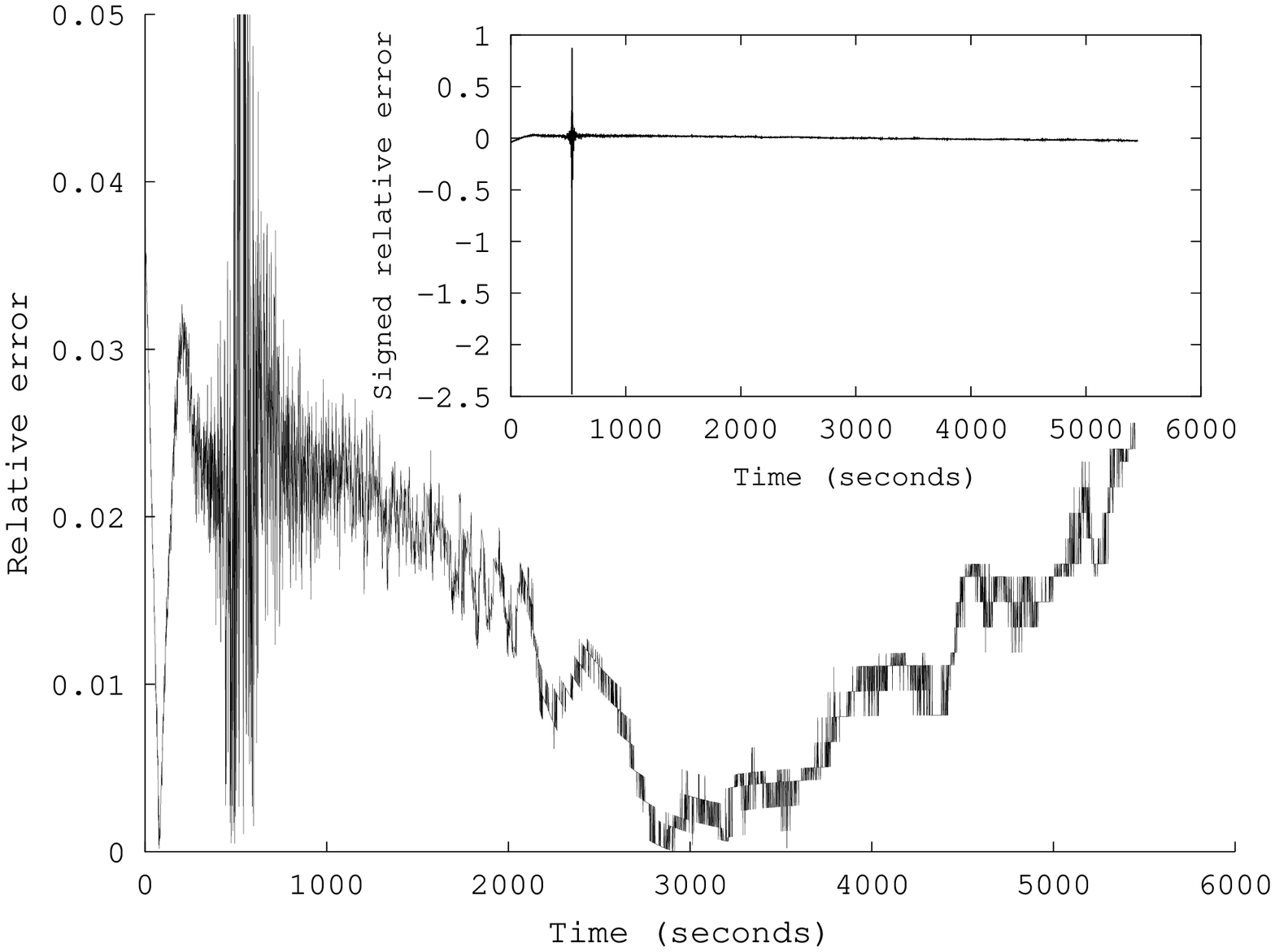}
    \caption{Relative error for values lower or equal than $0.05$. On a small scale, the signed relative error.}
      \label{subfig:errors_cooling}
  \end{subfigure}%
  \begin{subfigure}[t]{.5\textwidth}
    \centering
    \includegraphics[width=\linewidth]{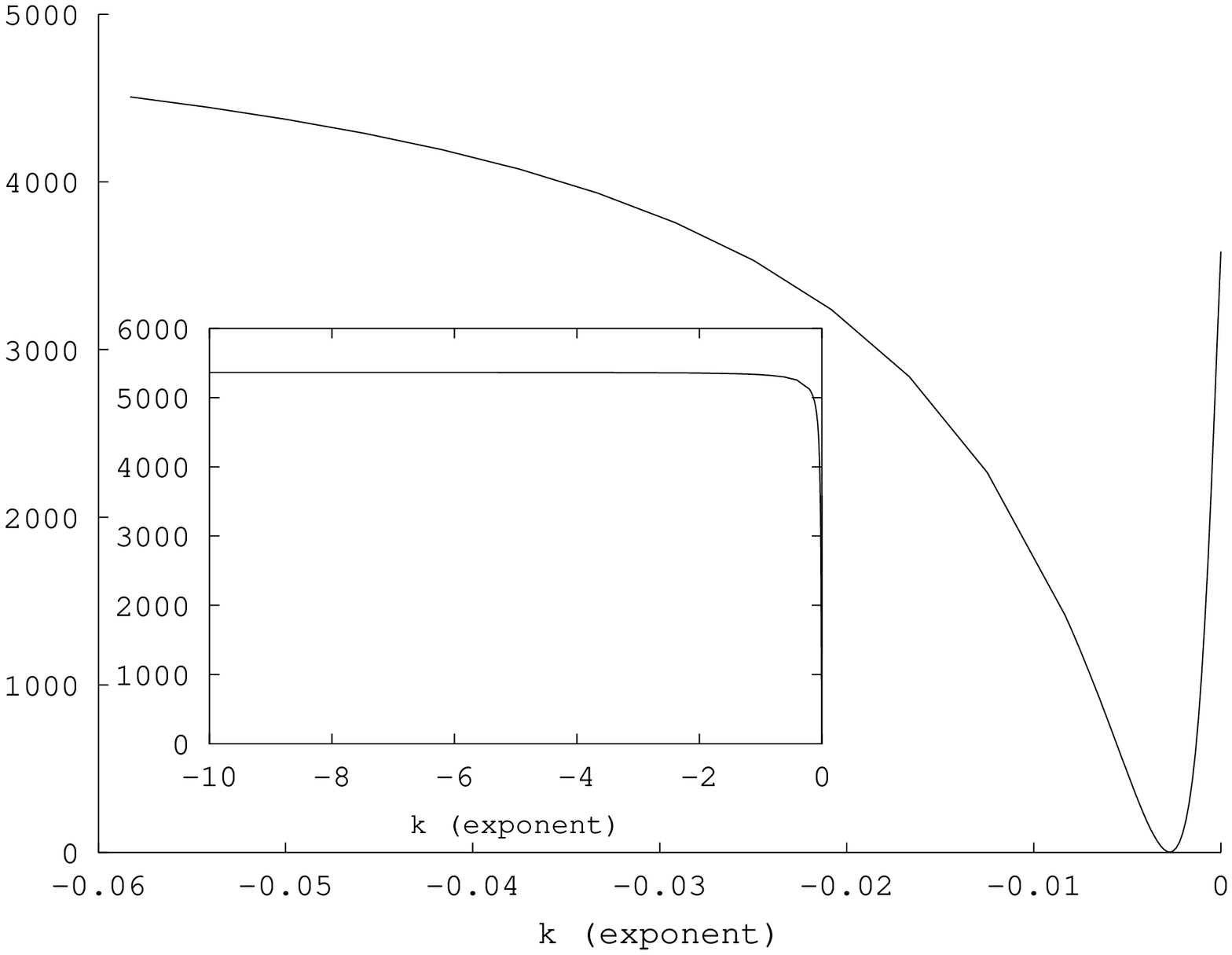}
    \caption{Quasiconvex shape of $\E2$.}
  \end{subfigure}%
  \caption{Some graphical aspects about this TAC implementation.}
\end{figure}

A spike can be seen in the small window of Figure \ref{subfig:errors_cooling}. This spike should not be considered as an indicator of a poor adjustment of the curve to the data. On the contrary: the spike is due to the proximity of the data to zero and, however, the error remains bounded. This is because curve and data are close enough to control the fact that we are virtually dividing by zero.

\section{Fitting biexponential decays in stress relaxation on living cells} \label{s_aplications_biexp}

Stress relaxation experiments are widely spread in the study of the viscoelastic mechanical properties of soft matter, such as cells, polymer brushes, vesicles\ldots~Indentation experiments are carried out by scanning probe microscopes, e.g., Atomic Force Microscopes. The tip of that microscope is placed in contact with the sample at a constant height and, if the sample shows a viscoelastic behaviour, an exponential decay of the needed force to keep the tip in the same place is observed, provided the contact area remains constant.

In some studies, see \cite[p.~3]{beniteztoca2010}, when the sample is formed by two materials with different (visco-)elastic properties, the force decay can be modelled by a double exponential decay in the form
\begin{equation}\label{ec_dos_exponenciales}
P(t)=\lambda_1 e^{k_1 t} + \lambda_2 e^{k_2 t} + \lambda_3, 
\end{equation}
being real all the parameters and $k_1, k_2 <0$.

We will show two different implementations. The first one, included for demonstration purposes only, is constituted by Table \ref{tbfit_14_curvas}. This table shows the behaviour of TAC fitting curves corresponding to the first 5 of the 14 examples that can be found in \cite{benitezbolos2017}. A huge interval, $[-20,-10^{-6}]$, containing $k_1$ and $k_2$ will be considered; $10$ will be chosen as the default mesh partition and the stop condition in this case will be fixed at $\alpha=10^{-6}$. The point of this table is to allow the reader to compare TAC against a widely known and used algorithm: the Levenberg-Marquardt algorithm, taking in mind that we do not need to initialize any variable whatsoever.

\begin{table}[ht]
\begin{center}
\resizebox{\textwidth}{!}{
\begin{tabular}{ccccccc}
\toprule
$k_1$ & $k_2$ & $\lambda_1$ & $\lambda_2$ & $\lambda_3$ & SSR & MSE \\
\midrule
$-4.7832922$ & $-0.3521602$ & $1.6849216e-10$ & $1.3747312e-10$ & $6.5258959e-10$ & $7.0319408e-20$ & $3.4352422e-23$ \\
$-2.2501220$ & $-0.0154859$ & $1.7351215e-10$ & $6.0038892e-10$ & $1.7367841e-10$ & $1.0648907e-19$ & $5.2022018e-23$ \\
$-4.1127733$ & $-0.1999089$ & $1.5888120e-10$ & $1.2906191e-10$ & $6.7388509e-10$ & $7.8436193e-20$ & $3.8317632e-23$ \\
$-6.6396453$ & $-0.3583037$ & $1.5563740e-10$ & $1.4391666e-10$ & $6.5265553e-10$ & $7.6892152e-20$ & $3.7563337e-23$ \\
$-4.9698671$ & $-0.3825276$ & $1.6264867e-10$ & $1.2442707e-10$ & $6.5888391e-10$ & $7.6859116e-20$ & $3.7547199e-23$ \\ \bottomrule
\end{tabular}
}
\vspace*{-0.2cm}
\caption{Parameters gathered in this table correspond to the biexponential decay of pattern (\ref{ec_dos_exponenciales}). Time of CPU processing is between 1.96 and 1.98 seconds for each curve.} \vspace{0.35cm}
\label{tbfit_14_curvas}
\end{center}
\end{table} 

We will go further in a second implementation, fitting data obtained in an experiment developed at the Institute for Biophysics, of the University of Natural Resources and Life Sciences (BOKU-Wien). These records present, not only a biexponential decay, but also some periodic signals which are probably due to the oscillations of the microscope's cantilever as a consequence of electronic noise. The pattern considered in this case is
\begin{equation}\label{ec_oscilante}
\begin{split}
P(t) & = \lambda_{1} e^{k_1 t} + \lambda_{2} e^{k_2 t} + \lambda_3 \\
& \quad + \beta_1 \sin(\mu_1 t)+ \beta_2 \cos(\mu_1 t)+ \beta_3 \sin(\mu_2 t)+ \beta_4 \cos(\mu_2 t),
\end{split}
\end{equation}
being real every parameter, $k_1,k_2<0$ and $\mu_1,\mu_2>0$.

We have chosen $[-20,-10^{-6}]$ as the interval where $k_1$ and $k_2$ can be found, $[10^{-6},10]$ as the interval where $\mu_1$ and $\mu_2$ can be found and fixed $\alpha=10^{-6}$ as the algorithm's stop condition.

Note that, as in previous section, we are going beyond of our knowledge about the convergence of this algorithm in order to obtain evidences about its robustness. In this case the function $\E2$ exhibits a deep and wide enough absolute minimun allowing the algorithm to work properly; even though the quasiconvexity of $\E2$ is not guaranteed. These characteristics of this minimum are the key of its robustness.

\begin{figure}[ht]
  \centering
  \begin{subfigure}{.5\textwidth}
    \centering
    \includegraphics[width=\linewidth]{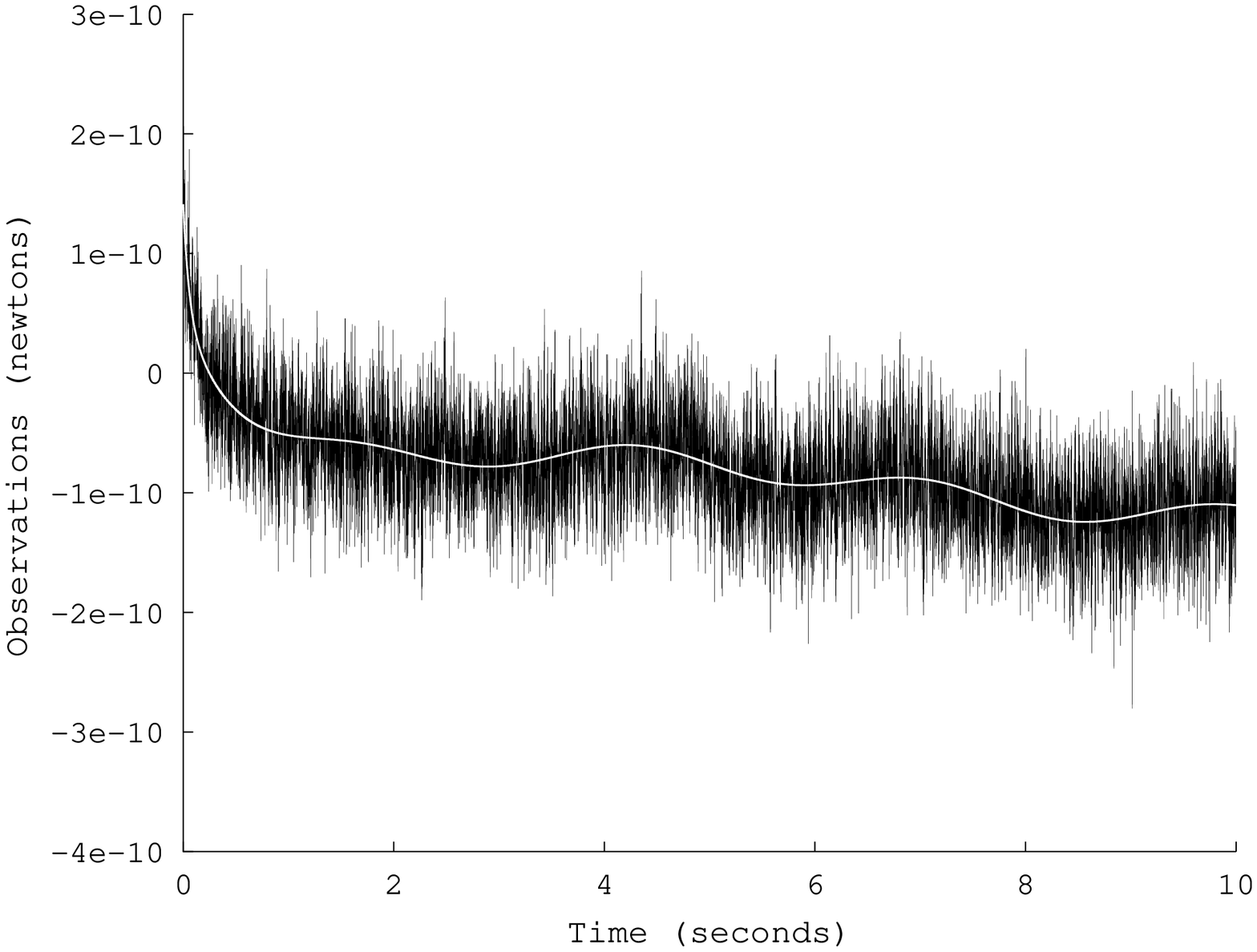}
    \caption{Data (black) and fit (white).}
    \end{subfigure}\hfill
  \begin{subfigure}{.5\textwidth}
    \centering
    \includegraphics[width=\linewidth]{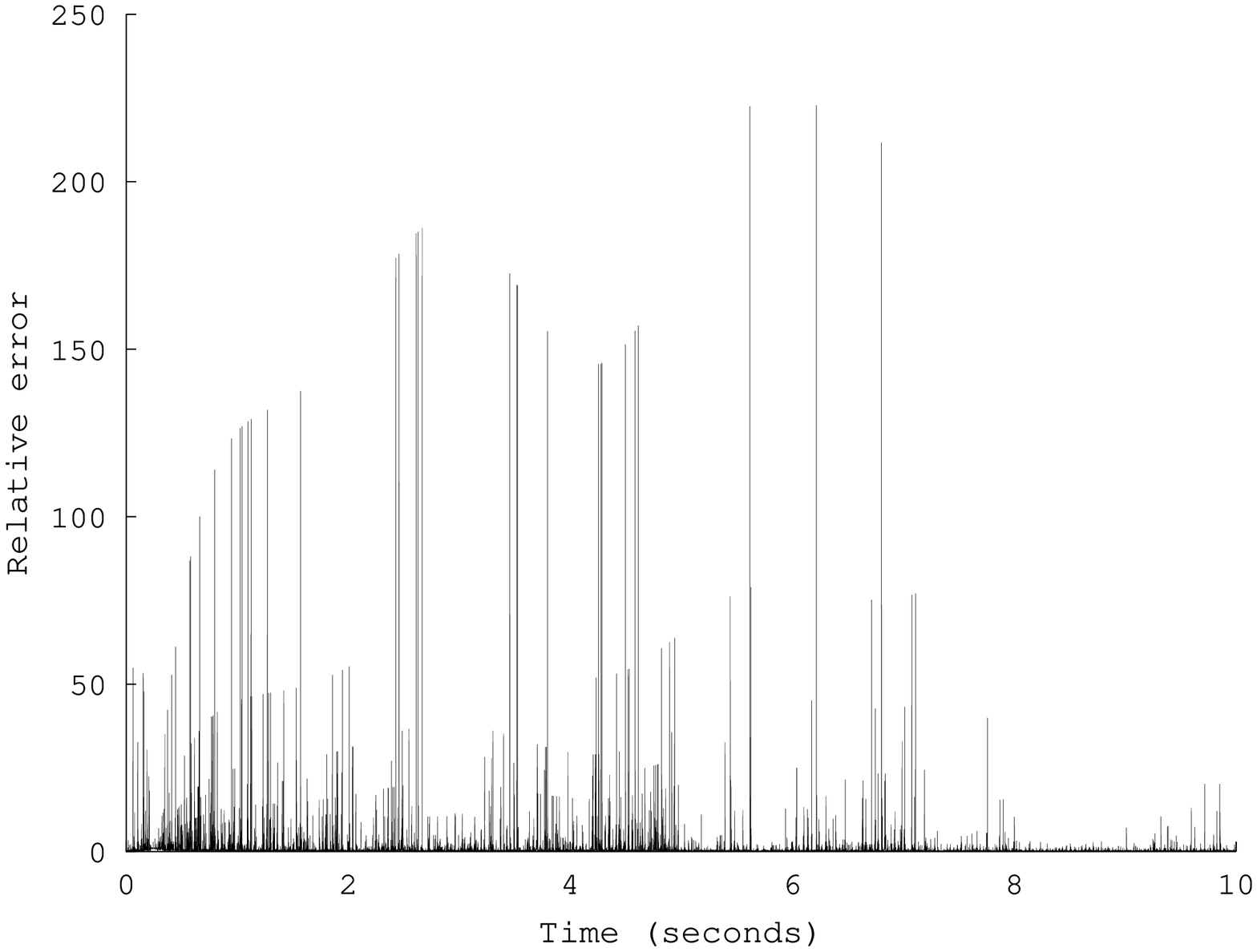}
    \caption{Relative error.}
    \label{subfig:errors_weird}
  \end{subfigure}%
  \caption{A word to the wise: don't miss the white line and the spikes\ldots}
\end{figure}

Figure \ref{subfig:errors_weird} shows that the relative errors are mostly around zero, although some spikes appear. These are numerous since the observations are quite noisy and many of them are close to zero. Even in this scenario, the spikes remain bounded, as far as we can tell, for the same reasons as in Section \ref{s_aj_cooling}.

\begin{table}
\begin{center}
\resizebox{\textwidth}{!}{
\begin{tabular}{ccccccc}
\toprule 
Divisions & CPU Time & $k_1$ & $k_2$ & $\lambda_1$ & $\lambda_2$ & $\lambda_3$ \\
& (in seconds) & & & & & \\ \midrule
$5$ & $135.93$ & $-4.9611570$ & $-1.0005773e-06$ & $1.7071912e-10$ & $7.3616932e-06$ & $-7.3617394e-06$ \\
$10$ & $2370.45$ & $-17.7513354$ & $-1.6802928$ & $9.2406742e-11$ & $1.5523536e-10$ & $-9.4023128e-11$ \\
$15$ & $12213.18$ & $-17.7513259$ & $-1.6802933$ & $9.2406730e-11$ & $1.5523535e-10$ & $-9.4023128e-11$ \\ \bottomrule
\end{tabular}
}
\vspace*{-0.2cm}
\caption{Parameters gathered in this table correspond to the biexponential decay of pattern (\ref{ec_oscilante}).} \vspace{0.35cm}
\end{center}
\end{table}

\begin{table}
\begin{center}
\resizebox{\textwidth}{!}{
\begin{tabular}{ccccccc}
\toprule 
Divisions & $\mu_1$ & $\mu_2$ & $\beta_1$ & $\beta_2$ & $\beta_3$ & $\beta_4$ \\ \midrule
$5$ & $2.1297350$ & $10$ & $5.9826879e-12$ & $-7.6271944e-12$ & $-3.0835996e-12$ & $2.0909088e-12$ \\
$10$ & $0.4874903$ & $2.2811038$ & $2.5556496e-11$ & $-2.5090025e-12$ & $-3.6852033e-12$ & $-9.4834037e-12$ \\
$15$ & $0.4874902$ & $2.2811036$ & $2.5556499e-11$ & $-2.5089862e-12$ & $-3.6851942e-12$ & $-9.4834075e-12$ \\ \bottomrule
\end{tabular}
}
\vspace*{-0.2cm}
\caption{Parameters gathered in this table correspond to the oscilatory part in pattern (\ref{ec_oscilante})} \vspace{0.35cm}
\end{center}
%
\begin{center}
\resizebox{.45\textwidth}{!}{
\begin{tabular}{ccc}
\toprule 
Divisions & SSR & MSE \\ \midrule
$5$ & $2.5877776e-17$ & $1.2636250e-21$ \\
$10$ & $2.5466615e-17$ & $1.2435478e-21$ \\
$15$ & $2.5466615e-17$ & $1.2435478e-21$ \\ \bottomrule
\end{tabular}
}\hspace*{.50\textwidth}
\vspace*{-0.2cm}
\caption{As in Table \ref{tbfit1}, $RSS=\sum_{i=1}^n (T_i-\F_k(t_i))^2$, being $n=20479$ the number of observations and $MSE=RSS/n$. TAC seems to be stabilized when taking $10$ or greater as the default mesh partition.} 
\label{tbfit2}
\end{center}
\end{table}

\section{Final Remarks}

Please observe that, allowing $k$ to range over $(0,\infty)$, we may apply verbatim everything we have done to the study of the remaining exponential evolutions. 

It would be more than welcome any hint about the conditions $T$ must fulfill to ensure the quasiconvexity of $\E2$. 

It is not too hard to adapt these methods to weighted norms, but every calculation must be carried out with care. 

\section{Acknowledgments}
We are grateful to José Luis Toca-Herrera for providing data to be fitted in Section \ref{s_aplications_biexp} and Rafael Benítez for his useful suggestions. We thank the crew on board RV Las Palmas for their inexhaustible spirit and enthusiasm shown during the Antarctic campaign. We are grateful to the military personnel at the Gabriel de Castilla Spanish Antarctic station for their support in the deployment and recovery of sensors; with special mention to Captain Javier Barba and Corporal María del Carmen Pesqueira.

Funding: This work was partially supported by MICINN [research project reference CTM2010-09635 (subprogramme ANT)], MINECO [research project reference MTM2016-76958-C2-1-P] and Consejería de Economía e Infraestructuras de la Junta de Extremadura [research projects references: IB16056 and GR15152].

\bibliographystyle{plain}
\addcontentsline{toc}{chapter}{Bibliography}
\bibliography{bibliografia_quasiconvex}{}

\end{document}